\documentclass[runningheads,a4paper]{llncs}

\usepackage{amssymb}
\usepackage{graphicx}
\usepackage{amsmath}
\usepackage{url}
\usepackage[left=3.5cm,right=3.5cm,
    top=4.4cm,bottom=4cm,bindingoffset=0cm]{geometry}
\urldef{\mailsa}\path|parolja@gmail.com|
\newcommand{\keywords}[1]{\par\addvspace\baselineskip
\noindent\keywordname\enspace\ignorespaces#1}
\makeatletter
\renewcommand*{\@fnsymbol}[1]{\ensuremath{\ifcase#1\or  	\star  \or \bullet \or   *\or \else\@ctrerr\fi}}
\makeatother
\begin{document}

\mainmatter  % start of an individual contribution

\title{On arithmetic index in the generalized Thue-Morse word\thanks{This work was performed within the framework of the LABEX MILYON (ANR-10-LABX-0070) of Universite de Lyon, within the program "Investissements d'Avenir" (ANR-11-IDEX-0007) operated by the French National Research Agency (ANR).}\thanks{ The publication is available at {\tt link.springer.com/chapter/10.1007/978-3-319-66396-8\_12}}}

\titlerunning{On arithmetic index in the generalized Thue-Morse word}

\author{Olga G. Parshina}
\authorrunning{O. Parshina}

\institute{Sobolev Institute of Mathematic SB RAS,\\
4 Acad. Koptyug avenue, 630090 Novosibirsk, Russia,\\
Universit\'{e} de Lyon, Universit\'{e} Claude Bernard Lyon 1, Institut Camille Jordan,\\
43 boulevard du 11 novembre 1918, F-69622 Villeurbanne Cedex, France\\
\mailsa\\}


\toctitle{On Arithmetic Index in the Generalized Thue-Morse Word}
\tocauthor{Olga~G.~Parshina}
\maketitle

\begin{abstract}
Let $q$ be a positive integer. Consider an infinite word $\omega=w_0w_1w_2\cdots$ over an alphabet of cardinality $q$. A finite word $u$ is called an arithmetic factor of $\omega$ if $u=w_cw_{c+d}w_{c+2d}\cdots w_{c+(|u|-1)d}$ for  some choice of positive integers $c$ and $d$. We call $c$ the initial number and $d$ the difference of $u$. For  each such $u$ we define its arithmetic index by $\lceil\log_q d\rceil$ where $d$ is the least positive integer such that $u$ occurs in $\omega$ as an arithmetic factor with difference $d$. In this paper we study the rate of growth of the arithmetic index of arithmetic factors of a generalization of the Thue-Morse word defined over an alphabet of prime cardinality. More precisely, we obtain upper and lower bounds for the maximum value of the arithmetic index in $\omega$ among all its arithmetic factors of length $n$.
\end{abstract}

\section{Introduction}
One of the main characteristics of a given word is the factor complexity which counts the number of its distinct factors of each fixed length. We are interested in studying of so-called arithmetic factors. In other words, for a given infinite word $\omega=w_0w_1w_2\cdots$ over a finite alphabet $\Sigma$ we are studying the structure of its {\it arithmetic closure} -- the set $A_\omega=\{w_c w_{c+d}w_{c+2d}\cdots w_{c+(n-1)d}|c\geq 0, d, n\geq 1\}$. Elements of $A_\omega$ are arithmetic subsequences or arithmetic factors with initial number $c$ and difference $d$ of the word $\omega$. Of special interest are arithmetic factors having period 1, which are called {\it arithmetic progressions}. According to the classical Van der Waerden theorem \cite{War}, the arithmetic closure of each infinite word $\omega$ over an alphabet of cardinality $q$ for every positive integer $n$ contains an arithmetic progression of length $n$.
A point of interest is to determine an upper bound on the minimal difference, with which the arithmetic progression of length $n$ appears in the arithmetic closure of a given word. The first result of the paper (see Theorem~\ref{theor}) concerns the distribution of words with period $1$ in case $\omega$ is an infinite word over an alphabet of prime cardinality generalising the classical Thue-Morse word originally introduced by Thue in \cite{Thue} (see also \cite{allsh}). For a prime $q$ and for every positive integer $n$ this theorem provides the maximal length of an arithmetic progression with difference $d<q^n$ in the generalized Thue-Morse word over the alphabet of cardinality $q$ and extends the earlier result on the generalized Thue-Morse word over the alphabet of cardinality 3 obtained by the author in \cite{Par}.

The next question appearing in this context concerns the distribution of arithmetic subsequences with period 2. In the case of binary alphabet such subsequences have period 01, and we call them alternating subsequences.
There is a conjecture, that in the Thue-Morse word alternating sequences are "the hardest to find", i.e.
if $d$ is the difference of the first occurrence of the alternating sequence of length $n$ in the Thue-Morse word as an arithmetic factor, and $h$ is the difference of the first arithmetic occurrence of some binary word of length $n$ in the Thue-Morse word, then $d$ is great or equal to $h$.

To carry out computer experiments and to check the conjecture we introduce the notion of arithmetic index. More precisely, given a positive integer $q$ and an infinite $q$-automatic word $\omega$, for every finite word $u$ from its arithmetic closure we seek to determine the least positive integer $d$ such that $u$ occurs in $\omega$ with difference $d$.
We call the $q$-ary expansion of the difference $d$ the {\it arithmetic index} of $u$ in $\omega$.

Computer experiments show that the set of words of the maximal arithmetic index in the Thue-Morse word contains alternating sequences, but they are not the sole members of this set. Describing this set even for a particular word did not appear to be an easy task.
In this paper we try to determine upper and lower bounds on the rate of growth of the arithmetic index in case when $\omega$ is the generalized Thue-Morse word over the alphabet of prime cardinality.
An upper bound is determined using the result on lengths of arithmetic progressions formulated in Theorem~\ref{theor}; a lower bound is obtained using the factor and arithmetical complexities of the word.

\section{Preliminaries}
Let $q$ be a positive integer and $\Sigma$ a finite alphabet of cardinality $q$. An infinite word over $\Sigma$ is an infinite sequence $\omega=w_0w_1w_2\cdots$ with $w_i\in\Sigma$ for every $i\in\mathbb{N}$. A finite word $u$ over $\Sigma$ is said to be a factor of $\omega$ if $u=w_jw_{j+1}\cdots w_{|u|+j-1}$ for some $j\in\mathbb{N}$.

For each positive integer $d$, let $A_\omega(d)=\{w_cw_{c+d}w_{c+2d}\cdots w_{c+(k-1)d}| c,k\in\mathbb{N}\}$ be the set of all arithmetic subsequences in the word $\omega$ of difference $d$. Elements of $A_\omega(d)$ are called arithmetic subwords or arithmetic factors of $\omega$.

The arithmetic closure of $\omega$ is the set $A_\omega=\bigcup\limits_{d=1}^\infty A_\omega(d)$ consisting of all its arithmetic factors, and the function $a_\omega(n)=|A_\omega\cap\Sigma^n|$ counting the number of distinct arithmetic factors of each fixed length $n$ occurring in $\omega$ is called the {\it arithmetical complexity} of $\omega$. The notion of arithmetical complexity was introduced by Avgustinovich,  Fon-der-Flaass and Frid in \cite{avg}. Since $A_\omega(1)$ coincides with the set of factors of $\omega$, it follows trivially that $a_\omega(n)\geq p_\omega(n)$. But aside from this basic inequality, there is no general relationship between the rates of growth of these two complexity functions.
For instance, there exist infinite words of linear factor complexity and whose arithmetical complexity grows linearly or exponentially, as seen in \cite{avg}; arithmetical complexity of Sturmian words, which have factor complexity equals $n+1$, grows as $O(n^3)$ (see \cite{fc}).  A characterization of uniformly recurrent words having linear arithmetical complexity one can see in \cite{af1}. The question about lowest possible complexity among uniformly recurrent words was studied in \cite{acf}.  A family of words with various sub-polynomial growths of arithmetical complexity was constructed in \cite{af2}.

For a given infinite word $\omega$ and a finite word $u\in A_\omega$ we are interested in the least positive integer $d$ such that $u$ belongs to $A_\omega(d)$. We denote the length of the $q$-ary representation of such a minimal difference as $i_\omega(u)$ and call this quantity the {\it arithmetic index} of $u$ in $\omega$. For each positive integer $n$, we consider the function $I_\omega(n) = \max\limits_{u\in A_\omega\cap \Sigma^n}i_\omega(u)$. Let us note, that this function is defined over the set of arithmetic factors of $\omega$.

In this we study the growth rate of the arithmetic index for a generalization of the Thue-Morse word defined over an alphabet $\Sigma_q = \{0,1,..., q-1\}$, where $q$ is a prime number. Let $S_q : \mathbb{N} \to \Sigma_q^+$ be the function which assigns to each natural number $x$  its base-$q$ expansion. The length of this word is denoted by $|S_q(x)|$. Also let $s_q(x)$ be the sum modulo $q$ of the digits in $q$-ary expansion of $x$. In other words, if $x=\sum\limits_{i=0}^{n-1} x_i\, q^i$, then  $S_q(x) = x_{n-1}\cdots x_1x_0$ and $s_q(x)= \sum\limits_{i=0}^{n-1} x_i$ mod $q$.
We define the  generalized Thue-Morse word $\omega_{q} = w_0w_1w_2w_3\cdots$ over the alphabet $\Sigma_q$ by  $w_i = s_q(i) \in \Sigma_q$. We note that this generalization differs from the one given in \cite{TrSh}. In case $q=2$, we recover the classical Thue-Morse word which is known to  be arithmetic universal, i.e. $a_{\omega_2}(n)=2^n$, as it is shown in \cite{avg}, moreover, using results of the paper it is easy to deduce that $a_{\omega_q}(n)=q^n$.
In case $q=3$,  the generalized Thue-Morse word over ternary alphabet is given by:
\begin{center} $\omega_{3}=012120201120201012201012120\cdots$ \end{center}
A lower and an upper bounds on the rate of growth of the function $I_{\omega_q}(n) = \max\limits_{u\in A_{\omega_q}\cap \Sigma_q^n}i_{\omega_q}(u)$ are obtained in the paper.
The upper bound grows as $O(n\log{n})$, the lower one grows linearly.

\section{Upper Bound on Arithmetic Index in $\omega_q$}
An upper bound is based on the distribution of arithmetic progressions -- arithmetic subsequences consisting of the same symbols -- in the generalized Thue-Morse word formulated below.

\subsection{Theorem on Arithmetic Progressions in $\omega_q$}
Let $L_\omega(c,d)$ be the function which outputs the length of an arithmetic progression with initial number $c$ and difference $d$ for positive integers $c$ and $d$ in an infinite word $\omega$. The function $L_\omega(d)=\max\limits_c L_\omega(c,d)$ gives the length of the maximal arithmetic progression with the difference $d$ in $\omega$.
Let us note, that for us the symbol of the alphabet on which the function $L_\omega(c,d)$ reaches its maxima is of no importance, since the set of arithmetic factors of the generalized Thue-Morse word is closed under adding a constant to each symbol.
\begin{theorem}
\label{theor}
  Let $q$ be a prime number and $\omega_q$ be the generalized Thue-Morse word over the alphabet $\Sigma_q$. For all integers $n\geq 1$ the following holds:
  \begin{equation*}
 \max_{d<q^n}L_{\omega_q}(d)   =
\begin{cases}
q^n+2q, &\text{$n\equiv 0$ mod $q$},\\
 q^n, &\text{otherwise.}
\end{cases}
\end{equation*}
Moreover, the maximum is reached with the difference $d=q^n-1$ in both cases.

\end{theorem}

\begin{proof}[of Theorem \ref{theor}]
Since the theorem is a generalization of the main result of \cite{Par}, the technique of proving is similar to one presented there.

As the first step it should be proved that for a fixed $n$ the inequality $d\neq q^n-1$ implies $L_{\omega_q}(d)\leq q^n$. During the proof we have to manipulate with values $q-1$ and $q-2$, thus let us use notations $\dot{q}:=q-1, \ddot{q}:=q-2$.

\subsubsection{Case of $d\neq q^n-1$. }

Let us note that subsequences of the $\omega_q$ which are composed of letters with indices having the same remainder of the division by $q$ are equivalent to the word itself, so we do not need to consider differences which are divisible by $q$.

\begin{lemma}
\label{lem0}
Let $q$ be a prime number and $\omega_q$ be the generalized Thue-Morse word over $\Sigma_q$. For any positive integer $n$ and $d\leq q^n-1$ the length of the longest arithmetic progression with difference $d$ in $\omega_q$ is not greater than $q^n$.
\end{lemma}
\begin{proof}
Every number can be represented in the following way: $c=y\,q^n + x$, where $x,y$ are arbitrary positive integers, $x<q^n$. Let us call $x$ the suffix of $c$.

Consider a set $X=\{0, 1, 2, ..., q^n-1\}$, its cardinality is $|X|=q^n$. As far as each difference $d$ and suffix $x$ belong to $X$ and $d$ is prime to $|X|$, the set $X$ is an additive cyclic group, and $d$ is a generator of $X$, thus for every $x\in X$ the set $\{x+i\, d\}_{i=0}^{q^n-1}$ is precisely $X$.
To proof the statement for this case, it is enough to provide for each $d\neq q^n-1$ an element $x\in X$ with the following properties:\\
(a)  $x+d<q^n$; \\
(b)  $s_q(x+d)\neq s_q(x)$.\\
Indeed, consider the initial number of the form $c=y\, q^n + x$ with $x$ satisfying (a) and (b) and $y$ being an arbitrary positive integer. Because of (a), $c+d=y\, q^n +(x+d)$.
Hence,  $s_q(c)=s_q(y)+s_q(x)$ mod $q$,  $s_q(c+d)=s_q(y)+s_q(x+d)$ mod $q$, and because of (b), $s_q(c+d)\neq s_q(c)$. That means, if we consider an arithmetic subsequence with difference $d$ starting with any symbol of generalized Thue-Morse word and having length $q^n+1$, then it will contain a symbol with the index of the form $c$ mentioned above and thus at least two different symbols of the alphabet $\Sigma_q$. This implies that the arithmetic progression in this case has length less or equal to $q^n$.

If $s_q(d)\neq 0$, then $x=0$ fits. In other case we use the inequation $d\neq q^n-1$ which means that $S_q(d)=d_{n-1}\cdots d_1d_0$ has at least one letter $d_j, j\in\{0,1,...,n-1\}$: $d_j\neq \dot{q}$. There are two possibilities:
\begin{enumerate}
\item  There exists at least one index $j$ such that $d_j < \dot{q}$ and $d_{j-1}= \dot{q}$.\\
 In this case $x=q^{j-1}$ fits. Indeed, the $q$-ary representation of $c$ is $S_q(y)S_q(x)$, where all symbols $x_i$ are zeros except $x_{j-1}=1$, thus $s_q(c)=s_q(y)+1$.
 The representation of the difference $d$ is $d_{n-1}\cdots d_j \dot{q} d_{j-2}\cdots d_0$, and the sum of its digits equals zero modulo $q$.
 More precisely, $\sum\limits_{i=0,i\neq j-1}^{n-1}{d_i} +\dot{q} \equiv 0$ mod $q$, or $\sum\limits_{i=0,i\neq j-1}^{n-1} d_i -1\equiv 0$ mod $q$. Once we add $d$ to $c$, we obtain the number $c+d$ having representation $S_q(y)d_{n-1}\cdots (d_j+1) 0 d_{j-2}\cdots d_0$, where $d_j+1\leq \dot{q}$. Then $s_q(c)=s_q(y)+s_q(d)+1- \dot{q}=s_q(y)+2$, which differs from $s_q(c)=s_q(y)+1$.\\
\item For every $j$ the fact $d_j\neq \dot{q}$ implies that all symbols having indices less than $j$ are not equal to $\dot{q}$. \\
If $j>0$, then $d_1,d_0\neq \dot{q}$, and $d_0\neq 0$ since $d$ is not divisible by $q$. In this case a suitable $x$ is $q-d_0$, because $s_q(x+d)=s_q(1-d_0)\neq s_q(q-d_0)=s_q(x)$.\\
But there is no $x$ satisfying (a) and (b) in the case $j=0$, i.e. then $S_q(d)=\underbrace{\dot{q}\cdots\dot{q}}_{\text{$n-1$}}d_0$.
However, we can take $x$ with $q$-ary expansion of the form $S_q(x)=x_{n-1}\cdots x_1 \dot{q}$, where $x_i\in\Sigma_q$ are arbitrary, and claim that for arbitrary value of $y$ we obtain the number with the sum of digits different from $s_q(c)=s_q(y)+s_q(x)$ after at most two additions of the difference. Consider these two steps. After adding to $c$ with $S_q(c)=S_q(y)\ x_{n-1}\cdots x_1\dot{q}$ and $s_q(c)=s_q(y)+\sum_{i=1}^{n-1}x_i+\dot{q}$ the difference of the form $\dot{q}\cdots \dot{q} d_0$ we obtain the number $c+d$ with $S_q(c+d)=S_q(y+1)x_{n-1}\cdots x_1\dot{d_0}$.
Its sum of digits is $s_q(y+1)+\sum_{i=1}^{n-1}x_i+\dot{d_0}$, and if it differs from $s_q(c)$, then this $x$ fits.
If the values $s_q(c)$ and $s_q(c+d)$ are equal, then the following holds: $s_q(y+1)+d_0\equiv s_q(y)$ mod $q$.
This implies that $q$-ary representation of $y$ ends with $0\underbrace{\dot{q}\cdots\dot{q}}_{\dot{q}-d_0}$, and thus $q$-ary representation of $y+1$ ends with zero.
After the next addition of the difference there are two cases.
If $2d_0\geq q$, we obtain $S_q(y+2)  x_{n-1}\cdots x_1 (\dot{2d_0})$ with $s(y+2)=s(y+1)+1$ and $s_q(c+2d)=s_q(y+1)+\sum_{i=1}^{n-1}x_i+2d_0$ mod $q$, which implies $d_0=\dot{q}$, but this is not the case.
If $2d_0<q$, we obtain the number of the form $S_q(c+2d)=S_q(y+2)  x_{n-1}\cdots x_2 \dot{x_1} (\dot{2d_0})$ with $s_q(c+2d)=s_q(y+1)+\sum_{i=1}^{n-1}x_i+2d_0-1$, which implies $d_0\equiv 0$ mod $q$ and contradicts the fact $d_0\neq 0$.
\end{enumerate}
Thus there are $q(n-1)$ different values for $x$, such that for every positive integer $c$ with the suffix $x$ either $s_q(c+d)\neq s_q(c)$, or $s_q(c+2d)\neq s_q(c+d)$. That means that every arithmetic progression with the difference $d$ with $S_q(d)=\underbrace{\dot{q}\cdots\dot{q}}_{\text{n-1}}d_0, d_0\neq0,d_0\neq\dot{q}$ is not be longer than $q^n$.

Since all possible values of difference $d\neq q^n-1$ are considered, the lemma is proved.
\end{proof}
\subsubsection{Case of $d=q^n-1$.}
We start with the following lemma.
\begin{lemma}
\label{lem1}
Let $q$ be a prime number and $\omega_q$ be the generalized Thue-Morse word over $\Sigma_q$.
Let $d=q^n-1$, $c =z\, q^{2n}+ y\, q^n+x$, where  $x+y=q^n-1$, $z$ is a non-negative integer,  then
\begin{equation*}
\max_z L_{\omega_q}(c, d) =
\begin{cases}
x+q+1, &\text{$n \equiv 0$ mod $q$},\\
   x+1, &\text{otherwise.}
\end{cases}
\end{equation*}
\end{lemma}
\begin{proof}

For descriptive reasons let us introduce a scheme where one can see base-$q$ expansions of $c+id$ and values of $s_q(c+id)$ for each value of $i$, and let us give some comments on that.
\begin{center}
\begin{tabular}{|c|r|c|}
\hline
$\mathbf{i}$&\ \ $\mathbf{S_q(c+id)}$\ \ \ \ \ \ \ \ \ \ \ \ \ \ \ \ \ \ \ \ & $\mathbf{s_q(c+id)}$\\
\hline
 0&\ \ \ $S_q(z)$ $y_{n-1}\ \cdots \ \ y_1y_0$ \ $x_{n-1}\cdots x_1x_0$\ \ \ &\ \  $n\dot{q}+s_q(z)$\ \ \\
 \vdots &\ \ \ \ \ \ \ \vdots \ \ \ \ \ \ \ \ \ \ \ \ \vdots\ \ \ \ \ \ \ \ \ \ \ \ \ \ \ \ \ \ \vdots\ \ \ \ \ \  &\vdots\\
 \vdots &\ \ \ \ \ \ \ \vdots \ \ \ \ \ \ \ \ $+x\, d$ \ \ \ \ \ \ \ \ \ \ \ \ \ \ \vdots\ \ \ \ \ \  & \vdots \\
 \vdots &\ \ \ \ \ \ \ \vdots \ \ \ \ \ \ \ \ \ \ \ \ \vdots\ \ \ \ \ \ \ \ \ \ \ \ \ \ \ \ \ \ \vdots\ \ \ \ \ \  &\vdots\\
$x$&\ \ \ $S_q(z)$ \ $\dot{q}\ \ \ \cdots \ \ \ \dot{q}\ \dot{q}$ \ \ \ \ $0\ \ \ \cdots \ 0\ 0$\ \ \ \  &\ \  $n\dot{q}+s_q(z)$\ \ \\
 $x+1$&\ \ \ $S_q(z)$ \ $\dot{q}\ \ \ \cdots \ \ \ \dot{q}\ \dot{q}$ \ \ \ \ $\dot{q}\ \ \ \cdots \ \dot{q}\ \dot{q}$\ \ \ \   &\ \  $2n\dot{q}+s_q(z)$\ \ \\
 $x+2$&\ $S_q(z+1)$ \ $0\ \ \ \cdots \ \ \ 0\ 0$ \ \ \ \ $\dot{q}\ \ \ \cdots \ \dot{q}\ \ddot{q}$\ \ \ \  &\ \  $n\dot{q}-1+s_q(z+1)$\ \ \\
 \vdots &\ \ \ \ \ \ \ \vdots \ \ \ \ \ \ \ \ \ \ \ \ \vdots\ \ \ \ \ \ \ \ \ \ \ \ \ \ \ \ \ \ \vdots\ \ \ \ \ \  &\vdots\\
 \vdots &\ \ \ \ \ \ \ \vdots \ \ \ \ $+(q-2)\, d$ \ \ \ \ \ \ \ \ \ \ \ \vdots\ \ \ \ \ \  &\vdots \\
 \vdots &\ \ \ \ \ \ \ \vdots \ \ \ \ \ \ \ \ \ \ \ \ \vdots\ \ \ \ \ \ \ \ \ \ \ \ \ \ \ \ \ \ \vdots\ \ \ \ \ \  &\vdots\\
 $x+q$&\ $S_q(z+1)$ \ $0\ \ \ \cdots \ \ \ 0\ \ddot{q}$ \ \ \ $\dot{q}\ \ \ \ \cdots \ \dot{q}\ 0$\ \ \ \  &\ \  $n\dot{q}-1 +s_q(z+1)$\ \ \\
\ $x+q+1$&\ $S_q(z+1)$ \ $0\ \ \ \cdots \ \ \ 0\ \dot{q}$ \ \ \ $\dot{q}\ \ \ \ \cdots \ \ddot{q}\ \dot{q}$\ \ \ \   &\ \  $n\dot{q}+\ddot{q} +s_q(z+1)$\ \ \\
\hline
\end{tabular}
\end{center}
Values in the third column are sums modulo $q$.

Since $d=q^n-1$, we can regard the action $c+d$ as two simultaneous actions: $x-1$ and $y+1$.
Thus, while the suffix of $c+i\,d$ is greater then zero, the sum of digits in $S_q(c+i\,d)$ equals $\dot{q}\,n$. This value holds during the first $x$ additions of $d$ (when $i=0,1,..,x$), and on the step number $x$ the length of the arithmetic progression is $x+1$.

On the next step ($i=x+1$) the sum of digits in result's $q$-ary representation becomes  $\,\, 2\,\dot{q}\,n+s_q(z)$. To preserve the required property of progression members we need $\,\,\dot{q}\,n\equiv2\,\dot{q}\,n$ mod $q$, i.e., $n\equiv 0$  mod $q$.

After the next addition of the difference, $z$ increases to $z+1$, $y$ becomes $0$, $x=q^n-2$ and the sum modulo $q$ of digits in this number $q$-ary expansion becomes $s_q(z+1)+s_q(x)=s_q(z+1)+n\,\dot{q}-1$. 
We may choose a suitable $z$ to hold the homogeneity of the progression, e.g. if $s_q(z)=1$ we need $s_q(z+1)=2$, and $z$ may be equal to $1$. The value of the sum modulo $q$ holds during the following $q-2$ editions of $d$, and we get into the situation of $y=q-2$, $x=q^n-\dot{q}$.

After the next addition $y$ becomes $\dot{q}$, and $x=q^n-q-1$.
Now $s_q(y)+s_q(x)=n\,\dot{q}+\ddot{q}$ mod $q$ and is not equal to its previous value $n\,\dot{q}-1$.

Hence, in the case of $q|n$ the length of an arithmetic progression is $x+q+1$ and it is $x+1$ otherwise. The lemma is proved.
\end{proof}

\begin{lemma}
\label{lem2}
Let $q$ be a prime number, and $\omega_q$ be the generalized Thue-Morse word over $\Sigma_q$.
Let $n\equiv 0$ mod $q$, $d=q^n-1$, $c = z\, q^{2n}+ y\, q^n+x$, $y=q^n-\dot{q}$, $x=\dot{q}$, and $z$ be an arbitrary non-negative integer, then $\max\limits_z L_{\omega_q}(c, d)=q^n+2\,q$.
\end{lemma}

\begin{proof}
The sum modulo $q$ of digits in $S_q(c)$ is equal to $s_q(z)+n\,\dot{q} +1$, and by arguments similar to the ones used in Lemma~\ref{lem1}, this value is not changing while the suffix of $c+i\,d$ is greater or equal to zero, i.e. during $\dot{q}$ steps; then we get into a situation when $y=x=0$ and $z$ is increased by 1. We may set $z$ to be a zero to hold the homogeneity of a progression on this step.

After the next addition we get into conditions of Lemma~\ref{lem1} with $x=q^n-1$, which provides us with an arithmetical progression of length $q^n+q$.

Now we subtract $d$ from the initial $c$ to make sure that $s_q(c-d)\neq s_q(c)$ and we cannot obtain longer arithmetical progression.
Indeed, $c-d=z\, q^{2n}+(q^n-\dot{q})\, q^n+q$ and the sum of digits in its $q$-ary representation is $\,\,n\,\dot{q}-\dot{q}+1$, while in $c$ it is $\,\,n\,\dot{q}+1$. Hence the length of this progression is $1+\dot{q}+q^n+q=q^n+2\,q$, and the lemma is proved.
\end{proof}

Now let us prove that we can not construct an arithmetical progression with the difference $d=q^n-1$  longer than $q^n+2q$.

Here we represent the initial number $c$ of the progression this way: $c=y\, q^n +x$, $x<q^n$.

The case of initial number $c$ with $x_j+y_j = \dot{q}$, $j=0,1,...,n-1$ is described in Lemma~\ref{lem1}. In other case there is at least one index $j$ such that $x_j+y_j\neq \dot{q}$. We choose $j$ which is the minimal. There are $q\,\dot{q}$ possibilities of values $(y_j, x_j)$: $(0,0), (0,1), ..., (0,\ddot{q}), (1,0),..., (\dot{q},\dot{q})$.

Integers $y$ and $x$ have $q$-ary representations $S_q(y)=y_{s-1}\cdots y_{j+l+1}\dot{q}\cdots\dot{q}y_j\cdots y_0$ and $S_q(x)=x_{n-1}\cdots x_{j+m+1}0\cdots0x_j\cdots x_0$, where $0\leq l\leq s-j$, $0\leq m\leq n-j$, $y_{j+l+1}\neq \dot{q}$, and $x_{j+m+1}\neq 0$.

We add $q^{j+1}\, d$ to $c$. If $l\neq0$, the block $\,\,\dot{q}\cdots\dot{q}\,\,$ in $S_q(y)$ transforms to the block of zeros; if $m\neq0$, the block of zeros in $S_q(x)$ transforms to the block $\,\,\dot{q}\cdots\dot{q}$, $y_{j+l+1}$ increases by one, and $x_{j+m+1}$  decreases by one. To hold the homogeneity we need $l$ and $m$ to be equal modulo $q$.

There are two different cases.
\begin{enumerate}
\item If $x_j<\dot{q}-y_j$, then after $(x_j+1)\, q^j$ additions of the difference we obtain the number with $q$-ary expansion $y_{s-1}\cdots y_{j+l+2}(y_{j+l+1}+1)0\cdots0(y_j+x_j+1)y_{j-1}\cdots y_0$ $x_{n-1}\cdots x_{j+m+2}\\ \dot{x_{j+m+1}} \dot{q}\cdots\dot{q}\ddot{q}\dot{q}x_{j-1}\cdots x_0$ with the sum of digits $s_q(y)+s_q(x)+\dot{q}\neq s_q(y)+s_q(x)=s_q(c)$. Thus the length of an arithmetic progression is not greater than $q^j(q+x_j+1)\leq q^n$ if $j<n-1$.

\item If $x_j>\dot{q}-y_j$, we add $(q-y_j)\, q^j\, d$ and obtain a number of the form $y_{s-1}\cdots y_{j+l+2}(y_{j+l+1}+1)0\cdots010 y_{j-1}\cdots y_0$ $x_{n-1}\cdots x_{j+m+2}\dot{x_{j+m+1}}\dot{q}\cdots\dot{q}(x_j -q+y_j)x_{j-1}\cdots x_0$ with the sum of digits $s_q(y)+1-l\dot{q}+1-y_j+ s_q(x)+m\dot{q}-1-q+y_j=s_q(y)+s_q(x)+1$. The length of an arithmetic progression in this case is not greater than $q^j(2\,q-y_j)\leq q^n$, if $j<n-1$.
\end{enumerate}

\begin{example} Let us consider an example for $q=5$, $m=l=3$, $n=6$, $j=1$,$(y_j,x_j)=(1,1)$, $d=15624$, $S_5(d)=444444$.
\begin{center}
\begin{tabbing}
11111111111111\=11111111111111111\=11111111111111111111\=1111111\kill\\
\>$\mathbf{number}$\> $\mathbf{S_5}$ \> $\mathbf{s_5}$ \\
\>\>\>\\
\>$c=97396881$\>\thinspace$144413200011$\>$1$\\
\>\>+\ \ \ \thinspace\thinspace$44444400$ \>\\
\>$c+25\, d$\> \thinspace $\overline{200013144411}$ \>$1$\\
\>\>+\ \ \ \ \ \thinspace $4444440$\>\\
\>$c+30\, d$ \>\thinspace $\overline{200023144401}$ \> $1$\\
\>\>+\ \ \ \ \ \thinspace $4444440$\>\\
\>$c+35\, d$ \> \thinspace $\overline{200033144341}$ \> $0$\\
\end{tabbing}
\end{center}
\end{example}
The case $j=n-1$ needs a special consideration.

The way of acting is the same: we add $\,\,x\,d\,\,$ to $\,c\,$ and nullify $x$ by that,
then add $d$ necessary number of times. Thus the worst case is then $S_q(x)=x_{n-1}\dot{q}\cdots\dot{q}$ and $S_q(y)=y_{s-1}\cdots y_n y_{n-1}0\cdots 0$. The length of an arithmetic progression is $x+2<q^n+2$ if $n\equiv0$ mod $q$ and $x+1\leq q^n$ otherwise. But since $y_{n-1}+x_{n-1}\neq\dot{q}$, there are two cases to consider, let us introduce schemes for both.\\
\begin{enumerate}
\item $r=y_{n-1}+x_{n-1}<\dot{q}$.
\begin{center}
\begin{tabular}{|c|r|c|}
\hline
 $\mathbf{i}$ & $\mathbf{S_q(c+id)}$\ \ \ \ \ \ \ \ \ \ \ \ \ \ \ \ \ \ \ \    &  $\mathbf{s_q(c+id)}$ \\
\hline
 0 &\ \ $y_{s-1}\cdots y_n\ \ y_{n-1}$\ \ $0\cdots 0$\ $x_{n-1}\ \dot{q}\cdots \dot{q}\ \dot{q}$\ \ &\ \  $n\dot{q}-\dot{q}+r+y_n+...+y_{s-1}$\ \ \\
 \vdots&\ \ \ \ \ \ \ \ \ \vdots \ \ \ \ \ \ \ \ \ \ \ \ \ \ \ \ \ \ \ \vdots\ \ \ \ \ \ \ \ \ \ \ \ \ \ \ \ \vdots\ \ \ \ \ \ \ \  &\vdots\\
 \vdots &\ \ \ \ \ \ \ \ \ \vdots \ \ \ \ \ \ \ \ \ \ \ \ \ \ \ $+x\, d$ \ \ \ \ \ \ \ \ \ \ \ \ \vdots\ \ \ \ \ \ \ \  &\vdots \\
 \vdots &\ \ \ \ \ \ \ \ \ \vdots \ \ \ \ \ \ \ \ \ \ \ \ \ \ \ \ \ \ \ \vdots\ \ \ \ \ \ \ \ \ \ \ \ \ \ \ \ \vdots\ \ \ \ \ \ \ \  &\vdots \\
$x$ &\ \ $y_{s-1}\cdots y_n\ \ \ \ r$\ \ \ \ \ $\dot{q}\cdots\dot{q}$\ \ \ \ $0\ \ \ 0\cdots0\ 0$\ \  &\ \ $n\dot{q}-\dot{q}+r+y_n+...+y_{s-1}$\\
 \ \ $x+1$\ \  &\ \ $y_{s-1}\cdots y_n\ \ \ \ r$\ \ \ \ \ $\dot{q}\cdots\dot{q}$\ \ \ \ $\dot{q}\ \ \ \dot{q}\cdots\dot{q}\ \dot{q}$\ \  &\ \ $2n\dot{q}-\dot{q}+r+y_n+...+y_{s-1}$\ \ \\
 \ \ $x+2$\ \ &\ \ $y_{s-1}\cdots y_n \tiny (r+1)$\ $0\cdots0$\ \ \ \  $\dot{q}\ \ \ \dot{q}\cdots\dot{q}\ \ddot{q}$\ \  &\ \ $n\dot{q}+r+y_n+...+y_{s-1}$\ \ \\
\hline
\end{tabular}
\end{center}

\item $r=y_{n-1}+x_{n-1}>\dot{q}$.
Here we denote by $y_{s-1}'\cdots y_n'$ symbols $y_{s-1}\cdots y_n$ transformed after increasing  $y_n$ by 1 on the step $x$; to hold the homogeneity their sum should be equal to $y_{s-1}+ ... + y_n$ mod $q$.
\begin{center}
\begin{tabular}{|c|r|c|}
\hline
 $\mathbf{i}$ & $\mathbf{S_q(c+id)}$ \ \ \ \ \ \ \ \ \ \ \ \ \ \ \ \ \ \ \ \  &\ \ $\mathbf{s_q(c+id)}$ \\
\hline
 0 &\ \ $y_{s-1}\cdots y_n\ \ \ \ y_{n-1}$\ \ \ \ \ \ $0\cdots 0$\ $x_{n-1}\ \dot{q}\cdots \dot{q}\ \dot{q}$ \ \ &\ \  $n\dot{q}-\dot{q}+r+y_n+...+y_{s-1}$\ \ \\
\vdots &\ \ \ \ \ \ \ \ \ \vdots \ \ \ \ \ \ \ \ \ \ \ \ \ \ \ \ \ \ \ \ \ \ \ \vdots\ \ \ \ \ \ \ \ \ \ \ \ \ \ \ \ \vdots\ \ \ \ \ \ \ \ \ &\vdots \\
 \vdots &\ \ \ \ \ \ \ \ \vdots \ \ \ \ \ \ \ \ \ \ \ \ \ \ \ \ \ \ \ $+x\, d$ \ \ \ \ \ \ \ \ \ \ \ \ \vdots\ \ \ \ \ \ \ \ \ &\vdots \\
 \vdots &\ \ \ \ \ \ \ \ \ \vdots \ \ \ \ \ \ \ \ \ \ \ \ \ \ \ \ \ \ \ \ \ \ \ \vdots\ \ \ \ \ \ \ \ \ \ \ \ \ \ \ \ \vdots\ \ \ \ \ \ \ \ \ &\vdots \\
$x$ &\ \ $y_{s-1}'\cdots y_n'\ \ (r-q)$\ \ \ \ \ $\dot{q}\cdots\dot{q}$\ \ \ \ $0\ \ \ 0\cdots0\ 0$ \ \ &\ \ $n\dot{q}-\dot{q}+r+y_n'+...+y_{s-1}'$\ \ \\
\ $x+1$\ &\ \ $y_{s-1}'\cdots y_n'\ \ (r-q)$\ \ \ \ \ $\dot{q}\cdots\dot{q}$\ \ \ \ $\dot{q}\ \ \ \dot{q}\cdots\dot{q}\ \dot{q}$ \ \ &\ $2n\dot{q}-\dot{q}+r+y_n'+...+y_{s-1}'$\         \\
\ $x+2$\ &\ \ $y_{s-1}'\cdots y_n'\ (r-q+1)$\ $0\cdots0$\ \ \ \  $\dot{q}\ \ \ \dot{q}\cdots\dot{q}\ \ddot{q}$ \ \ &\ \ $n\dot{q}+r+y_n'+...+y_{s-1}'$\ \ \\
\hline
\end{tabular}
\end{center}
\end{enumerate}
Hence all possible cases have been considered and the theorem is proved.
\end{proof}

\subsection{Upper Bound on the Arithmetic Index for $\omega_q$}
Let $q$ be a prime, $n$ be a positive integer, and $u$ be a finite word over the alphabet $\Sigma_q$ of length $m$, where $q^{n-1}\leq m< q^n$.
The goal is to find its occurrence in $\omega_q$ as an arithmetic factor.
To reach the goal we use an arithmetic factor of $\omega_q$ of the form $0^{m-1}\beta_1\cdots\beta_{m}$, where each $\beta_i\in\Sigma_q$ and $\beta_1\neq0$. Lemma~\ref{lem2} provides us with its initial symbol $c$ and difference $d=q^n-1$.

Then we define a basis $\{b_i\}_{i=1}^m$
\begin{tabbing}
\=111\=1111111111111111111111111\kill\\
\>\>$b_1=\ 0\ \ 0\ \ 0\ \ \cdots\ \ \ \ 0\ \ \ \ \beta_1;$\\
\>\>$b_2=\ 0\ \ 0\ \ 0\ \ \cdots\ \ \ \beta_1\ \ \ \beta_2;$\\
\>\>$\ \ \ \vdots\ \ \ \ \ \ \ \ \ \ \ \ \ \ \ \ \vdots\ \ \ \ \ \ \ \ \ \ \ \ \ \ \vdots$\\
\>\>$b_m=\beta_1\beta_2\beta_3\ \cdots\ \beta_{m-1}\ \beta_m$,
\end{tabbing}
where all basis elements are arithmetic factors of $\omega_q$ with the difference $d$, and their initial numbers are $c_1=c$, $c_{i+1}=c+i\, d, i=1,...,m-1$.
The word $u$ can be represented in the following form: $u=\bigoplus\limits_{i=1}^{m} \alpha_i\, b_i$ with $\alpha_i\in\Sigma_q$ for every $i=1,2,...,m$, where $\bigoplus$ means symbol-to-symbol addition modulo $q$.

 Then let us construct the initial number $c_u$ of the arithmetic factor $u$ with $q$-ary representation $S_q(c_u)=\underbrace{S_q(c_1)\cdots S_q(c_1)}_{\alpha_1}\cdots\underbrace{S_q(c_m)\cdots S_q(c_m)}_{\alpha_m}$. According to Lemma~\ref{lem2} the length of each $S_q(c_i)$ is not greater than $2\,n+q$. To simplify the process let us put zeros left to the nonzero symbol with the maximal index in every $S_q(c_i)$ when it is necessary. Hence the length of $S_q(c_u)$ is not greater than $\,(2\,n+q)\,m\dot{q}$. Then we construct the difference $d_u$ with $S_q(d_u)=\underbrace{0\cdots0}_{\text{n+q}}\underbrace{\dot{q}\cdots\dot{q}}_{\text{n}} \cdots\underbrace{0\cdots0}_{\text{n+q}}\underbrace{\dot{q}\cdots\dot{q}}_{\text{n}}$ of the same length, and the word $u$ is guaranteed to appear in $\omega_q$ as an arithmetic factor with difference $d_u$ and initial number $c_u$.

The worst case is when all basis elements in the representation of $u$ are taken with coefficients $\dot{q}$.
In this case the length of $S_q(d_u)=(2\,n+q)\,\dot{q}\, m =(2\lceil\log_q{m}\rceil+q)\dot{q}\, m$,
 which is the upper bound on the function of arithmetic index in $\omega_q$. Thus we proved the following inequality:
\begin{equation*}
I_{\omega_q}(m) \leq (2\lceil\log_q{m}\rceil+q)\dot{q}\, m.
\end{equation*}

\section{Lower Bound on Arithmetic Index}

Consider the set $A_\omega(d)$ of all arithmetic subsequences with the difference $d$ in the word $\omega$ over
the alphabet $\Sigma_q$. Define a function $A_\omega(d,m)=|A_\omega(d)\cap\Sigma_q^m|$ counting the number
of different arithmetic factors with difference $d$ and of length $m$ in $\omega$.
Clearly, $|A_\omega(1)\cap\Sigma_q^m|=p_\omega(m)$ and, more general, $A_\omega(d,m)=|A_\omega(d)\cap\Sigma_q^m|\leq p_\omega(d\, m)$.

Obtaining a lower bound on the function $I_{\omega}(m)$ is equivalent to
obtaining the lower bound on $x$ in the following inequality:
\begin{equation*}
 a_\omega(m)\leq \sum\limits_{d=1}^{x} A_\omega(d,m)\leq \sum\limits_{d=1}^{x} p_\omega(d\, m).
\end{equation*}
%Thus $x+\frac{1}{2} \geq \sqrt{\frac{1}{4}+\frac{2a_\omega(m)}{p_\omega(m)}}$

The Thue-Morse word and its generalization are fixed points of uniform morphisms, their factor complexity is known to grow linearly \cite{Ehr}, i.e. $p_{\omega_q}(m)\leq C\, m$ for a positive integer $C$. As mentioned in sect.~2, $a_{\omega_q}(m)=q^m$, thus the inequality for the generalized Thue-Morse word takes the following form:

\begin{equation*}
 q^m \leq \sum\limits_{d=1}^{x} C\,d\,m=\frac{C\,m\,x(x+1)}{2}.
\end{equation*}
Solving this inequality we obtain a lower bound on $x$, which is $x\geq \sqrt{\frac{1}{4}+\frac{2q^m}{Cm}} -\frac{1}{2}$.

The 
We are interested in the integer part of $\log_q x$. The value $\lceil\log_q(x+\frac{1}{2})\rceil$ is either $\lceil\log_q{x}\rceil$ or $\lceil\log_q{x}\rceil+1$, and $\lceil\log_q x\rceil +1 \geq \lceil\log_q(x+\frac{1}{2})\rceil \geq \Big\lceil\log_q{\sqrt{\frac{1}{4}+\frac{2q^m}{Cm}}}\,\Big\rceil$.\\

%Thus $I_{\omega}(m) \geq \lceil\log_q x\rceil \geq \lceil\frac{1}{2}\log_q\frac{2a_{\omega}(m)}{p_\omega(m)} -1\rceil$

The lower bound on the function of arithmetic index in $\omega_q$ is
\begin{equation*}
 I_{\omega_q}(m)\geq \lceil\log_q x\rceil \geq \bigg\lceil\log_q\sqrt{\frac{1}{4}+\frac{2q^m}{Cm}}\bigg\rceil\geq \bigg\lceil\frac{1}{2}\log_q\frac{2q^m}{Cm}\bigg\rceil,
\end{equation*}
and the leading term of this expression has linear order.

The factor complexity of $\omega_2$ was computed in 1989 by Brlek \cite{b} and de Luca \& Varricino \cite{LV}. Using their result we can set the constant $C$ to be 4 in this case, and the lower bound can be written in the following way:
\begin{equation*}
I_{\omega_2}(m)\geq \frac{m-\log_2 m-1}{2}.
\end{equation*}

\section{Conclusion}
To sum up, the upper bound on the function of arithmetic index grows as $O(m\log m)$, the lower bound grows as $O(m)$. 
Given arbitrary infinite word with factor and arithmetic complexities are known, one can easily compute a lower bound on the arithmetic index in this word.
The upper bound is more difficult to compute and requires deeper knowledge of the word structure.

According to computer experiments, which were carried out for the Thue-Morse sequence,
the real growth of the function $I_{\omega_2}(m)$ is closer to the lower bound. Moreover, both theoretical reasoning and computer data show, that alternating arithmetic subsequences have the maximal arithmetic index. But they are not the only subsequences contained in the set of words with extremal arithmetic index; this set is going to be described for $\omega_q$ and then for other automatic words.

\end{document}